\documentclass[12pt]{amsart}
\usepackage{amsmath}
\usepackage{amssymb}
\usepackage{amsthm}
\usepackage[shortlabels]{enumitem}
\usepackage{color}
\usepackage{tikz}
\usepackage{tikz-qtree}
\usepackage{verbatim}

\newtheorem{proposition}{Proposition}[section]
\newtheorem{corollary}[proposition]{Corollary}
\newtheorem{theorem}[proposition]{Theorem}
\newtheorem{lemma}[proposition]{Lemma}

\theoremstyle{definition}

\newtheorem{remark}[proposition]{Remark}
\newtheorem{example}[proposition]{Example}

\newcommand{\insN}{\mathbb{N}}

\newcommand{\Spec}{\mathrm{Spec}}
\newcommand{\Max}{\mathrm{Max}}

\newcommand{\Crit}{\mathrm{Crit}}

\DeclareMathOperator{\suc}{Sc}
\DeclareMathOperator{\li}{Lim}

\newcommand{\XX}{\mathbf{X}}

\title{A realization theorem for almost Dedekind domains}
\author{Balint Rago}
\author{Dario Spirito}
\address{University of Graz, NAWI Graz, Department of Mathematics and Scientific Computing, Heinrichstraße 36,
8010 Graz, Austria}

\thanks{This work was supported by the Austrian Science Fund FWF, Project Number W1230}

\email{balint.rago@uni-graz.at}
\address{Dipartimento di Scienze Matematiche, Fisiche e Informatiche, Universit\`a di Udine, Udine, Italy}
\email{dario.spirito@uniud.it}
\date{\today}

\begin{document}

\maketitle

\begin{abstract}
   An integral domain $D$ is called an SP-domain if every ideal is a product of radical ideals. Such domains are always almost Dedekind domains, but not every almost Dedekind domain is an SP-domain. The SP-rank of $D$ provides a natural measure of the deviation of $D$ from being an SP-domain.  In the present paper we show that every ordinal number $\alpha$ can be realized as the SP-rank of an almost Dedekind domain.
\end{abstract}

\section{Introduction}
Dedekind domains are characterized by the property that any ideal has prime factorization, that is, it can be factorized as a (finite) product of prime ideals. Weakening this condition, one can ask about domains where every ideal has radical factorization, i.e., where every ideal is a (finite) product of radical ideals: such domains are called \emph{SP-domains} and, while not every SP-domain is Dedekind, every such domain is an \emph{almost Dedekind domain}, i.e., it is locally a discrete valuation ring (or, equivalently, it is locally a Dedekind domain). SP-domains are often nicer than general almost Dedekind domains: for example, it is possible to express the group of invertible ideals of an SP-domain as a group of continuous functions from its maximal space \cite[Theorem 5.1]{HK-Olb-Re}.

SP-domains can be characterized in several different ways among the almost Dedekind domains (see \cite[Theorem 2.1]{olberding-factoring-SP} or \cite[Theorem 3.1.2]{fontana-factoring}). To understand the properties of almost Dedekind domains that are not SP-domains, the paper \cite{SP-scattered} introduced a chain $\{\Delta_\alpha\}_\alpha$ of subsets of the maximal space $\Max(D)$ (which also defines a chain $\{T_\alpha\}_\alpha$ of overrings of $D$) that allows to study such domains in more depth; see Section \ref{sect:prelim:almded} for the definition. The smallest ordinal number $\alpha$ such that $\Delta_\alpha=\emptyset$ (or, equivalently, such that $T_\alpha=K$) measures the deviation of the domain from having radical factorization, and is called the \emph{SP-rank} of $D$. The chain $\{\Delta_\alpha\}_\alpha$ and the SP-rank of $D$ are formally analogous, respectively, to the derived series and to the Cantor-Bendixson rank of a topological space.

SP-domains are precisely the almost Dedekind domains that have SP-rank $1$. Examples 3.4.1 and 3.4.2 of \cite{fontana-factoring} are examples of almost Dedekind domains that are not SP-domains; in both cases, the first set $\Delta_1$ of the chain $\{\Delta_\alpha\}$ is a singleton, and thus the domains have SP-rank $2$. In this paper, we use these examples as a starting point of a construction that shows that every ordinal number $\alpha$ is the SP-rank of an almost Dedekind domain.

Our approach will be as follows. We define a graph $\mathcal{T}$ with vertex set $\mathbf{X}$, where the vertices are indexed by a pair $(\beta,t)$ where $\beta<\alpha$ is an ordinal number and $t$ is a sequence of ordinal numbers with certain properties (depending on $\beta$); we call $\beta$ the \emph{height} of the vertex. Next, we interpret the vertices of height $0$ as independent indeterminates over a field $F$, while the vertices of higher height are connected with them through infinite products, and we construct the domain $F[\mathbf{X}]$. We show that the divisibility properties of elements in this domain are closely linked to combinatorial properties of $\mathcal{T}$, more specifically to the structure of paths from vertices of height $\beta$ to vertices of height $0$; in particular, any two monomials (where a monomial is a finite product of elements of $\mathbf{X}$) have a greatest common divisor (Proposition \ref{prop:gcd monomial}). This allows us to define a multiplicative subset $S\subseteq F[\mathbf{X}]$ such that the localization $D:=S^{-1}F[\mathbf{X}]$ is a domain where every element is associated to a monomial: we explicitly determine the set of maximal ideals of $D$, and we show that a maximal ideal $M$ is in $\Delta_\beta$ if and only if it is generated by certain elements of height at least $\beta$: in particular, the chain $\{\Delta_\beta\}_\beta$ stabilizes precisely at $\alpha$, proving that $D$ has SP-rank $\alpha$.

An alternate construction is given in our concurrent paper \cite{almded-tree}, where we construct almost Dedekind domains of prescribed SP-rank from extensions of a discrete valuation ring; however, the method in \cite{almded-tree} only works for \emph{countable} ordinal numbers, while the one in the present paper gives almost Dedekind domains of arbitrary SP-rank.

\section{Preliminaries}

\subsection{Ordinals} Let $\alpha$ be an ordinal number. Then, $\alpha$ is a successor ordinal if there is a $\beta$ such that $\alpha=\beta+1$, while it is a limit ordinal otherwise. We denote by $\omega$ the first limit ordinal, i.e., the limit of the finite ordinals.

Given an ordinal $\alpha$, we denote by $\suc(\alpha)$ the set of all successor ordinals $\beta<\alpha$, and by $\li(\alpha)$ the set of all limit ordinals $\beta<\alpha$.

If $\xi,\alpha$ are ordinal numbers, we denote by $\xi^\alpha$ the set of all sequences $t=(t_0,t_1,t_2,\ldots,t_\omega,t_{\omega+1},\ldots)$ indexed by the ordinal numbers $\beta<\alpha$ such that $t_\beta<\xi$ for all $\beta$. We set $\alpha\setminus\{0\}$ to be the set of all ordinal numbers $0<\beta<\alpha$, and so $\xi^{\alpha\setminus\{0\}}$ is defined as above, but the elements of the sequence are indexed by the $\beta$ such that $0<\beta<\alpha$. In particular, if $\alpha=n$ is a finite ordinal, then the elements of $\xi^n$ are sequences $(t_0,\ldots,t_{n-1})$ of length $n$, while $\xi^{n\setminus\{0\}}$ contains sequences $(t_1,\ldots,t_{n-1})$ of length $n-1$.

If $0<\beta<\alpha$, we denote by $e_\beta\in\xi^{\alpha\setminus\{0\}}$ the element such that
\begin{equation*}
(e_{\beta})_\gamma=\begin{cases}
    1 & \text{if } \beta=\gamma \\
    0 & \text{otherwise}.
\end{cases}
\end{equation*}

\subsection{Graphs}
A \emph{directed graph} $\mathcal{T}$ is a pair $(V,E)$, where $V$ is a set (whose elements are called \emph{vertices}) and $E$ is a subset of $V\times V$ (whose elements $(v_1,v_2)$ are called \emph{edges}). In general, we use $\mathcal{T}$ to denote both the graph and the set of its vertices. A \emph{path} in $\mathcal{T}$ is a well-ordered sequence $(v_i)_{i\in I}$ of vertices such that $(v_i,v_{i+1})$ is an edge for all $i\in I$. We call $v_1$ the \emph{starting point} or the \emph{start} of the path.

\subsection{Almost Dedekind domains}\label{sect:prelim:almded}
An integral domain $D$ is said to be an \emph{almost Dedekind domain} if $D_M$ is a discrete valuation ring for every $M\in\Max(D)$; an almost Dedekind domain is always a one-dimensional Pr\"ufer domain.

A maximal ideal $M$ of $D$ is said to be \emph{critical} if it does not contain any finitely generated radical ideal, or equivalently, if every finitely generated ideal $I\subseteq M$ is contained in the square of a maximal ideal. We denote by $\Crit(D)$ the set of critical maximal ideals of $D$; this set is empty if and only if $D$ is an SP-domain (see \cite[Theorem 2.1]{olberding-factoring-SP} or \cite[Theorem 3.1.2]{fontana-factoring}).

We define recursively a chain $\{\Crit^\alpha(D)\}_\alpha$ of subsets of $\Max(D)$ and a chain $\{T_\alpha\}_\alpha$ of overrings of $D$ in the following way, where $\alpha$ is an ordinal number:
\begin{itemize}
\item $T_0:=D$, $\Crit^0(D):=\Max(D)$;
\item if $\alpha=\beta+1$ is a successor ordinal, then
\begin{equation*}
\Crit^\alpha(D):=\{P\in\Max(D)\mid PT_\beta\in\Crit(T_\beta)\};
\end{equation*}
\item if $\alpha$ is a limit ordinal, then
\begin{equation*}
\Crit^\alpha(D):=\bigcap_{\beta<\alpha}\Crit^\beta(D);
\end{equation*}
\item $\displaystyle{T_\alpha:=\bigcap\{D_M\mid M\in\Crit^\alpha(D)\}}$.
\end{itemize}
The set $\{\Crit^\alpha(D)\}_\alpha$ is a descending chain of subsets of $\Max(D)$, while $\{T_\alpha\}_\alpha$ is an ascending chain; we call the latter the \emph{SP-derived sequence} of $D$. Moreover, the maximal ideals of $T_\alpha$ are exactly the extensions of the maximal ideals in $\Crit^\alpha(D)$ \cite[Lemma 5.3]{SP-scattered}.

The \emph{SP-rank} of $D$ is the smallest ordinal number $\alpha$ such that $\Crit^\alpha(D)=\emptyset$; this rank always exists \cite[Theorem 5.1]{boundness}, and it is equal to the smallest ordinal number $\alpha$ such that $T_\alpha=K$, where $K$ is the quotient field of $D$.

\section{The graph}
From now on, let $\alpha$ be a fixed ordinal, and let $\xi:=\max\{\omega,\alpha\}$: we consider the set $\xi^{\alpha\setminus\{0\}}$. For every ordinal number $\beta<\alpha$, we define a function
\begin{equation*}
\tau_\beta:\xi^{\alpha\setminus \{0\}}\to\xi^{\alpha\setminus \{0\}},
\end{equation*}
where 
\begin{equation*}
(\tau_\beta(t))_\gamma:=\begin{cases}
t_\gamma & \text{if } \gamma\geq\beta+1 \\
0 & \text{otherwise.}
\end{cases}.
\end{equation*}
Hence $\tau_\beta$ preserves every component that is indexed by an ordinal $\gamma>\beta$ and maps every other component to zero.

For each $\beta\in\suc(\alpha)\cup \{0\}$, we define $U_\beta$ to be the set of all $t\in\xi_\alpha^{\alpha\setminus \{0\}}$ with only finitely many nonzero components, such that $t_\gamma\in \mathbb{N}_0$ for every $\gamma\in\suc(\alpha)$ and such that $t_{\delta}=0$ for all $\delta<\beta$.

If $\beta\in \li(\alpha)$, then we define $U_\beta$ in the exact same way, with the additional requirement that $t_\beta< \beta$ for all $t\in U_\beta$. 

\begin{example}
We explicitly write down the elements of $U_\beta$ when $\alpha=\omega+1$ and $\beta<\alpha$. We have $\suc(\alpha)=\{1,2,\ldots,n,\ldots,\}=\insN$ and $\li(\alpha)=\{\omega\}$.

Every element of $U_\beta$ is a sequence $(t_1,t_2,\ldots,t_n,\ldots,t_\omega)$ such that:
\begin{itemize}
\item if $\beta=n$ is a successor ordinal, then $t=(0,\ldots,0,t_n,\ldots,t_{\omega})$, where $t_i$ is a natural number for $i\in\insN$, $i\geq n$, while $t_\omega\in\insN_0\cup\{\omega\}$; moreover, only finitely many $t_i$ are nonzero;
\item if $\beta=\omega$, then $t=(0,0,\ldots,0,\ldots,t_\omega)$ with $t_\omega<\omega$ (i.e., $t_\omega\in\insN_0$).
\end{itemize}
\end{example}

\bigskip

We now construct a directed graph $\mathcal{T}$ in the following way. 
\begin{itemize}
    \item $\mathcal{T}$ has $\{X_{\beta,t}:\beta<\alpha,t\in U_\beta\}$ as its set of vertices.
    
    \item Let $\beta\in\suc(\alpha)$, say $\beta=\gamma+1$, and let $t\in U_\beta$. Then $t+ke_\beta\in U_\gamma$ and we set $(X_{\beta,t},X_{\gamma,t+k e_\beta})$ to be an edge for every $k\in \mathbb{N}$.
    
    \item Let $\beta\in \li(\alpha)$, $t\in U_\beta$ and let $\gamma\in\suc(\beta)$ be a successor ordinal such that $t_\beta<\gamma$. Then $\tau_\beta(t)\in U_\gamma$ and we set $(X_{\beta,t},X_{\gamma,\tau_\beta(t)})$ to be an edge.
    
    \item These are all the edges in $\mathcal{T}$.
\end{itemize} 

For a fixed ordinal $\beta<\alpha$, we denote the set of all $\{X_{\beta,t}:t\in U_\beta\}$ by $\mathbf{X}_\beta$ and we set \[\mathbf{X}:=\bigcup_{\beta<\alpha}\mathbf{X}_\beta.\] Moreover, we say that $X\in \mathbf{X}$ is of \textit{height} $\beta$ if $X\in \mathbf{X}_\beta$.

\begin{lemma}
Preserve the notation above. Then, the following hold:
\begin{enumerate}[(a)]
    \item every path in $\mathcal{T}$ is finite;
    \item if $X\in\mathbf{X}\setminus\mathbf{X}_0$, there is an $X'\in\mathbf{X}_0$ such that there is a path from $X$ to $X'$.
\end{enumerate}
\end{lemma}
\begin{proof}
If $(X_{\beta_i,s_i})_{i\in I}$ is a path, then $(\beta_i)_{i\in I}$ is a strictly descending chain of ordinal numbers. By well-ordering, such a chain must be finite.

Let $X=X_{\beta,t}$: we proceed by induction on $\beta$. By definition, there is an edge $(X_{\beta,t},X_{\beta',t'})$ for some $\beta'<\beta$; by induction, either $\beta'=0$ (and we are done), or there is a path from $X_{\beta',t'}$ to $X'\in\mathbf{X}_0$. Hence there is also a path from $X$ to $X'$, as claimed.
\end{proof}

\begin{lemma}\label{lemma:path-X0s}
Let $\beta$ be an ordinal, $t\in U_\beta,s\in U_0$, and suppose that there is a path from $X_{\beta,t}$ to $X_{0,s}$. Then, the following hold.
\begin{enumerate}[(a)]
\item\label{lemma:path-X0s:taubeta} $\tau_\beta(t)=\tau_\beta(s)$.
\item\label{lemma:path-X0s:sbeta} If $\beta$ is a successor ordinal, then $s_\beta>t_\beta$.
\end{enumerate}
\end{lemma}
\begin{proof}
Let $(X_{\beta,t}=X_{\beta_0,t_0},X_{\beta_1,t_1},\ldots,X_{\beta_n,t_n}=X_{0,s})$ be a path.

\ref{lemma:path-X0s:taubeta} We prove by induction that $\tau_\beta(t)=\tau_\beta(t_i)$.The fact that $\tau_\beta(t)=\tau_\beta(t_0)$ is obvious; suppose that $\tau_\beta(t)=\tau_\beta(t_{i-1})$. If $\beta_{i-1}$ is a successor ordinal, then $t_i=t_{i-1}+ke_{\beta_{i-1}}$, and in particular, the elements of $t_i$ and $t_{i-1}$ of index at least $\beta_{i-1}+1<\beta$ are the same; similarly, if $\beta_{i-1}$ is a limit ordinal, then $t_i=\tau_{\beta_{i-1}}(t_{i-1})$, and again all elements of index at least $\beta_{i-1}+1<\beta$ are the same. In both cases, $\tau_\beta(t_i)=\tau_\beta(t_{i-1})=\tau_\beta(t)$, and thus $\tau_\beta(s)=\tau_\beta(t_n)=\tau_\beta(t)$.

\ref{lemma:path-X0s:sbeta} Suppose that $\beta=\gamma+1$ is a successor ordinal. Then $n\ge 1$, $t_1=t+ke_\beta$ for some $k>0$ and thus $(t_1)_\beta> t_\beta$. By (1), we obtain that $\tau_{\beta_1}(t_1)=\tau_{\beta_1}(s)$ and since $\beta_1<\beta$, also that $(t_1)_\beta=s_\beta$. Hence $s_\beta>t_\beta$ and the claim is proved.
\end{proof}

\begin{lemma}\label{lemma:tree}
  Let $X,Y\in\mathbf{X}$. Then there is at most one path from $X$ to $Y$.
\end{lemma}

\begin{proof}
    Let $\beta<\alpha$, $t\in U_\beta$ and $s\in U_0$. Since every path in $\mathcal{T}$ can be extended to a path to some element of $\mathbf{X}_0$, it is enough to show that there is at most one path from $X_{\beta,t}$ to $X_{0,s}$.
This is clearly true if $\beta=0$. We proceed by induction. \\

Let $\beta$ be a successor ordinal, say $\beta=\gamma+1$, and let $k\in\mathbb{N}$ such that $(X_{\beta,t}, X_{\gamma,t+ke_\beta})$ is an edge and such that there is a path from $X_{\gamma,t+ke_\beta}$ to $X_{0,s}$. By Lemma \ref{lemma:path-X0s}\ref{lemma:path-X0s:taubeta}, we have that $\tau_\gamma(t+ke_\beta)=\tau_\gamma(s)$ and thus $t_{\beta}+k=s_{\beta}$. Hence the value of $k$ is uniquely determined by $s$ and since by induction hypothesis, there is a unique path from $X_{\gamma,t+ke_\beta}$ to $X_{0,s}$, we are done. 

Let now $\beta$ be a limit ordinal, $t\in U_\beta,s\in U_0$ and suppose that there are $\gamma_1,\gamma_2\in\suc(\beta)$ satisfying $t_\beta<\gamma_1<\gamma_2$, such that $(X_{\beta,t},X_{\gamma_1,\tau_\beta(t)})$ and $(X_{\beta,t},X_{\gamma_2,\tau_\beta(t)})$ are edges and such that there is a path from $X_{\gamma_1,\tau_\beta(t)}$ to $X_{0,s}$ and a path from $X_{\gamma_2,\tau_\beta(t)}$ to $X_{0,s}$. Since $\gamma_2$ is a successor ordinal, say $\gamma_2=\delta+1$, there is a $k\in \mathbb{N}$ such that there is a path from $X_{\delta,\tau_\beta(t)+ke_{\gamma_2}}$ to $X_{0,s}$. We then have $\tau_{\delta}(\tau_\beta(t)+ke_{\gamma_2})=\tau_{\delta}(s)$ and thus $s_{\gamma_2}=t_{\gamma_2}+k=k$. However, since there is a path from $X_{\gamma_1,\tau_\beta(t)}$ to $X_{0,s}$, we also have that $\tau_{\gamma_1}(\tau_\beta(t))=\tau_{\gamma_1}(s)$ and hence $s_{\gamma_2}=t_{\gamma_2}=0$, which yields a contradiction.

Hence there is at most one $\gamma\in\suc(\beta)$ with $t_\beta<\gamma$ such that there is a path from $X_{\gamma,\tau_\beta(t)}$ to $X_{0,s}$. Then by induction hypothesis, there is a unique path from $X_{\gamma,\tau_\beta(\gamma)}$ to $X_{0,s}$. 
\end{proof}

\section{The domain $F[\mathbf{X}]$}

Let $F$ be a field such that $\mathbf{X}_0$ is a set of independent indeterminates over $F$. For any $X\in\mathbf{X}\setminus\mathbf{X}_0$, we set \[X:=\prod_{(X,Y)\in\mathcal{T}}Y^2.\] Hence if $\beta\in \suc(\alpha)$, where $\beta=\gamma+1$, and $t\in U_\beta$, we have \[X_{\beta,t}=\prod_{i\in\mathbb{N}}X^2_{\gamma,t+ie_\beta}\] and
if $\beta\in\li(\alpha)$ and $t\in U_{\beta}$, we have \[X_{\beta,t}=\prod_{t_\beta<\gamma\in\suc(\beta)}X^2_{\gamma,\tau_\beta(t)}.\] 

\begin{remark}
    Defining $X_{\beta,t}$ in this way to be an infinite product is essentially equivalent to prescribing the following algebraic relations to elements of $\XX$. If $\beta\in\suc(\alpha)$ with $\beta=\gamma+1$ and $t\in U_\beta$, then \[X_{\beta,t}=X_{\gamma,t+e_\beta}^2\cdot X_{\beta,t+e_\beta}.\] If $\beta$ is a limit ordinal, we have \[X_{\beta,t}=X_{t_\beta+1,\tau_\beta(t)}^2\cdot X_{\beta,t+e_\beta}\] and if $\delta$ is a limit ordinal with $t_\beta<\delta<\beta$, then \[X_{\beta,t}=X_{\delta,\tau_\beta(t)+t_\beta e_\delta}\cdot X_{\beta,\tau_\beta(t)+\delta e_\beta}.\] Later on, we will see these identities in more detail. 
\end{remark} 

We will now investigate properties of the integral domain $F[\mathbf{X}]$.

We call an element $Y\in F[\mathbf{X}]$ a \textit{monomial} if \[Y=Y_1\ldots Y_n,\] where $n\in\mathbb{N}_0$ and $Y_1,\ldots, Y_n\in\mathbf{X}$. Note that the algebraic relations between elements in $\mathbf{X}$ are purely multiplicative. Hence $F[\mathbf{X}]$ is an $F$-vector space with the set of all monomials as basis. Moreover, every divisor of a monomial is itself a monomial. Alternatively, one can view $F[\mathbf{X}]$ as the monoid algebra $F[M]$, where $M$ is the multiplicative monoid generated by the elements in $\mathbf{X}$.

\begin{lemma}\label{lemma:product}
Let $\beta$ be an ordinal and let $t\in U_\beta$. Let $R_0:=R_0(X_{\beta,s}):=\{s\in U_0$ such that there is a path from $X_{\beta,t}$ to $X_{0,s}\}$. Then, there are positive integers $n_r$ such that
\begin{equation*}
X_{\beta,t}=\prod_{r\in R_0}X_{0,r}^{n_r}.
\end{equation*}
In particular, $n_r=2^{d(X_{\beta,t},X_{0,s})}$, where $d(X_{\beta,t},X_{0,s})$ is the length of the unique path from $X_{\beta,t}$ to $X_{0,s}$.
\end{lemma}
\begin{proof}
We proceed by induction on $\beta$. If $\beta=0$ then the claim is trivial. Suppose that the claim is true for ordinal numbers strictly smaller than $\beta$, and let $\mathcal{Y}$ be the set of all indeterminates $Y$ such that $(X,Y)$ is an edge. By definition and by induction, we have
\begin{equation*}
\begin{aligned}
X_{\beta,t}=\prod_{Y\in\mathcal{Y}}Y^2& =\prod_{Y\in\mathcal{Y}}\left(\prod_{r\in R_0(Y)}X_{0,r}^{n_r(Y)}\right)^2=\\
&=\prod_{Y\in\mathcal{Y}}\prod_{r\in R_0(Y)}X_{0,r}^{2n_r(Y)}
\end{aligned}
\end{equation*}
since if $(X,Y)$ is an edge then $Y=X_{\gamma,t'}$ with $\gamma<\beta$. If $Y\neq Y'$, we have $R_0(Y)\cap R_0(Y')=\emptyset$, because otherwise we would have two different paths from $X_{\beta,t}$ to $X_{0,s}$ (one through $Y$ and the other through $Y'$), against Lemma \ref{lemma:tree}; in particular, we have
\begin{equation*}
X_{\beta,t}=\prod_{r\in\bigcup_Y R_0(Y)}X_{0,r}^{n_r}
\end{equation*}
where $n_r=2n_r(Y)$ if $r\in R_0(Y)$. Moreover, any path from $X_{\beta,t}$ to $X_{0,s}$ must pass through (exactly) one $Y$: hence $R_0=\bigcup_YR_0(Y)$, and the existence of the product decomposition is proved.

Moreover, if the path from $X$ to $X_{0,r}$ passes through $Y$ and since $(X,Y)$ is an edge, $d(X,X_{0,r})=d(Y,X_{0,r})+1$; thus also the characterization of $n_r$ follows.
\end{proof}

\begin{lemma}\label{lemma:X0s-prime}
Let $s\in U_0$, and let $\{Z_\lambda\mid\lambda\in\Lambda\}\subseteq\XX$ be a family of indeterminates such that $Z:=\prod\{Z_\lambda\mid\lambda\in\Lambda\}\in F[\XX]$. Then:
\begin{enumerate}[(a)]
\item\label{lemma:X0s-prime:infprod} if $X_{0,s}$ divides $Z$, then it divides $Z_\lambda$ for some $\lambda\in\Lambda$;
\item\label{lemma:X0s-prime:prime} $X_{0,s}$ is a prime element of $F[\XX]$.
\end{enumerate}
\end{lemma}
\begin{proof}
\ref{lemma:X0s-prime:infprod} Consider the vector space $D$ over $F$ generated by the infinite products $\prod_{r\in R}X_r^{n_r}$, where $R$ ranges among the subsets of $U_0$ and $n_r\in\insN$ is arbitrary. Then, $D$ has a natural ring structure, that makes it into an integral domain, and by Lemma \ref{lemma:product} there is a natural embedding $F[\XX]\longrightarrow D$; thus, we can consider $F[\XX]$ to be a subring of $D$.

Write each $Z_\lambda$ as a product $\prod\{X_{0,r}^{n_r(\lambda)}\mid r\in R_0(\lambda)\}$. If $X$ divides $Z$ in $F[\XX]$, it also divides $Z$ in $D$; therefore, there must be a $\lambda\in\Lambda$ such that $r\in R_0(\lambda)$. By definition of $R_0(\lambda)$, it follows that there is a path from $X_{0,s}$ to $Z_\lambda$; by Lemma \ref{lemma:path->divides} this implies that $X_{0,s}$ divides $Z_\lambda$ in $F[\XX]$.

\ref{lemma:X0s-prime:prime} Suppose that $X_{0,s}$ divides a product $YZ$ with $Y,Z\in F[\XX]$, and write $Y=\sum a_iY_i$, $Z=\sum_j b_jZ_j$ where each $Y_i$ and each $Z_j$ are distinct monomials. Then, $X_{0,s}$ divides $YZ$ in $D$, and since $X_{0,s}$ is prime in $D$, we can suppose without loss of generality that $X_{0,s}$ divides $Y$ in $D$; hence, $X_{0,s}$ divides each $Y_i$ in $D$. By the previous part of the proof, it follows that $X_{0,s}$ divides $Y_i$ also in $F[\XX]$, and thus $X_{0,s}$ divides $Y$ in $F[\XX]$. Hence $X_{0,s}$ is prime in $F[\XX]$.
\end{proof}

\begin{lemma}\label{lemma:path->divides}
Let $X,Y\in \mathbf{X}$. If there is a path from $X$ to $Y$ in $\mathcal{T}$, then $Y^2$ divides $X$ in $F[\mathbf{X}]$. 
\end{lemma}

\begin{proof}
We start by showing that if $(X,Y)$ is an edge, then $Y^2$ divides $X$.

Let $\beta\in\suc(\alpha)$, where $\beta=\gamma+1$, let $t\in U_\beta$ and $k\in\mathbb{N}$. Then $(X_{\beta,t},X_{\gamma,t+ke_\beta})$ is an edge and we have \[
X_{\beta,t}=\prod_{i\in\mathbb{N}}X^2_{\gamma,t+ie_\beta}=\prod_{i=k+1}^{\infty}X_{\gamma,t+ie_\beta}^2\cdot \prod_{i=1}^{k}X_{\gamma,t+ie_\beta}^2.
\] Since\[X_{\beta,t+ke_\beta}=\prod_{i=k+1}^{\infty}X_{\gamma,t+ie_\beta}^2\] by definition and since $X_{\gamma,t+ke_\beta}^2$ clearly divides \[\prod_{i=1}^{k}X_{\gamma,t+ie_\beta}^2,\] it consequently also divides $X_{\beta,t}$ in $F[\mathbf{X}]$. \\

Let now $\beta\in \li(\alpha)$, $t\in U_\beta$ and let $\gamma\in\suc(\beta)$ such that $t_\beta<\gamma$. Moreover, we will write $\gamma=\delta+k$ for some $\delta\in\li(\alpha)\cup\{0\}$ and $k\in \mathbb{N}$. Then $(X_{\beta,t},X_{\gamma,\tau_\beta(t)})$ is an edge and we have \[X_{\beta,t}=\prod_{t_\beta<\varepsilon\in\suc(\beta)}X_{\varepsilon,\tau_\beta(t)}^2=X_{\gamma,\tau_\beta(t)}^2\cdot\prod_{t_\beta<\varepsilon\in\suc(\gamma)}X_{\varepsilon,\tau_\beta(t)}^2\cdot\prod_{\gamma<\varepsilon\in\suc(\beta)}X_{\varepsilon,\tau_\beta(t)}^2.\] 
By definition, we have \[X_{\beta,\tau_\beta(t)+\gamma e_\beta}=\prod_{\gamma<\varepsilon\in\suc(\beta)}X_{\varepsilon,\tau_\beta(t)}^2\] and thus it is enough to show that \[\prod_{t_\beta<\varepsilon\in\suc(\gamma)}X_{\varepsilon,\tau_\beta(t)}^2\in F[\mathbf{X}].\] If $\gamma<t_\beta+\omega$, then \[\prod_{t_\beta<\varepsilon\in\suc(\gamma)}X_{\varepsilon,\tau_\beta(t)}^2\] is a finite product of elements of $\mathbf{X}$ and hence contained in $F[\mathbf{X}]$.
If $t_\beta+\omega<\gamma$, then we can write \[\prod_{t_\beta<\varepsilon\in\suc(\gamma)}X_{\varepsilon,\tau_\beta(t)}^2=\prod_{i=1}^{k-1}X_{\delta+i,\tau_\beta(t)}^2\cdot\prod_{t_\beta<\varepsilon\in\suc(\delta)}X_{\varepsilon,\tau_\beta(t)}^2.\] 
Note that $t_\delta=0$ by definition of the set $U_\beta$ and hence \[\tau_\delta(\tau_\beta(t)+t_\beta e_\delta)=\tau_\delta(\tau_\beta(t))=\tau_\beta(t).\] This implies that \[X_{\delta,\tau_\beta(t)+t_\beta e_\delta}=\prod_{t_\beta<\varepsilon\in\suc(\delta)}X_{\varepsilon,\tau_\delta(\tau_\beta(t)+t_\beta e_\delta)}^2=\prod_{t_\beta<\varepsilon\in\suc(\delta)}X_{\varepsilon,\tau_\beta(t)}^2\] is contained in $F[\mathbf{X}]$ and we are done.   \\

Hence we have shown that if $(X,Y)$ is an edge in $\mathcal{T}$, then $Y^2$ divides $X$ in $F[\mathbf{X}]$. Since every path in $\mathcal{T}$ is finite, we can then state that the existence of a path from $X$ to $Y$ implies that $Y^2$ divides $X$.
\end{proof}

\begin{lemma}\label{lemma:X0s-divides}
Let $\beta$ be an ordinal and let $t\in U_\beta$, $s\in U_0$. Then, $X_{0,s}$ divides $X_{\beta,t}$ if and only if there is a path from $X_{\beta,t}$ to $X_{0,s}$.
\end{lemma}
\begin{proof}
By Lemma \ref{lemma:path->divides}, if such a path exists, then $X_{0,s}$ divides $X_{\beta,t}$. If $X_{0,s}$ divides $X_{\beta,t}$, the claim follows by writing $X_{\beta,t}=\prod_{r\in R_0}X_r^{n_r}$ and applying Lemma \ref{lemma:X0s-prime}\ref{lemma:X0s-prime:infprod}.


\end{proof}

\begin{remark}
Lemma \ref{lemma:X0s-divides} does not hold if we take some indeterminate $X_{\gamma,s}$ with $\gamma\neq 0$. Indeed, if $\gamma=\delta+1$, then $X_{\gamma,s}=X_{\gamma,s+e_s}X^2_{\delta,s+e_s}$, so that $X_{\gamma,s+e_s}$ divides $X_{\gamma,s}$, but there is no path from $X_{\gamma,s}$ to $X_{\gamma,s+e_s}$, since two indeterminates in a path cannot have the same height.
\end{remark}

\begin{lemma}\label{lemma:divide larger}
Let $0<\beta<\gamma<\alpha$ be ordinals and let $t,\Tilde{t}\in U_\beta$ and $s\in U_\gamma$.
\begin{enumerate}[(a)]
    \item\label{lemma:divide larger:beta} $X_{\beta,t}$ divides $X_{\beta,\Tilde{t}}$ if and only if $\tau_\beta(t)=\tau_\beta(\Tilde{t})$ and $t_\beta\ge \Tilde{t}_\beta$. In this case, $X_{\beta,\Tilde{t}}/X_{\beta,t}$ can be written as a product of elements in $\mathbf{X}$ of height less than $\beta$.
    \item\label{lemma:divide larger:gamma} $X_{\gamma,s}$ does not divide $X_{\beta,t}$.
\end{enumerate}
\end{lemma}

\begin{proof}
\ref{lemma:divide larger:beta} Suppose first that $\tau_\beta(t)=\tau_\beta(\Tilde{t})$ and $t_\beta\ge\Tilde{t}_\beta$. If $\beta$ is a successor ordinal, say $\beta=\delta+1$, then we have \[X_{\beta,\Tilde{t}}=X_{\beta,t}\cdot\prod_{i=1}^{t_\beta-\Tilde{t}_\beta}X_{\delta,\Tilde{t}+ie_\beta}^2\] by definition. Hence $X_{\beta,t}$ divides $X_{\beta,\Tilde{t}}$ in $F[\mathbf{X}]$ and $X_{\beta,\Tilde{t}}/X_{\beta,t}$ is a product of elements in $\mathbf{X}_\delta$.

If $\beta$ is a limit ordinal, then we have \[X_{\beta,\Tilde{t}}=\prod_{\Tilde{t}_\beta<\varepsilon\in\suc(\beta)}X_{\varepsilon,\tau_\beta(\Tilde{t})}^2\] and \[X_{\beta,t}=\prod_{t_\beta<\varepsilon\in\suc(\beta)}X_{\varepsilon,\tau_\beta(t)}^2.\] Since $\tau_\beta(t)=\tau_\beta(\Tilde{t})$ and $t_\beta\ge\Tilde{t}_\beta$, we obtain \[X_{\beta,\Tilde{t}}/X_{\beta,t}=\prod_{\Tilde{t}_\beta<\varepsilon\in\suc(t_\beta+1)}X_{\varepsilon,\tau_\beta(t)}^2.\] If $t_\beta<\Tilde{t}_\beta+\omega$, then this product is finite and thus contained in $F[\mathbf{X}]$. Moreover, it is a product of indeterminates of height less than $\beta$. Otherwise we can write $t_\beta=\delta+k$, where $\delta\in\li(\beta)\cup\{0\}$, $\Tilde{t}_\beta<\delta$, $k\in\mathbb{N}_0$ and we obtain \[\prod_{\Tilde{t}_\beta<\varepsilon\in\suc(t_\beta+1)}X_{\varepsilon,\tau_\beta(t)}^2=\prod_{i=1}^{k}X_{\delta+i,\tau_\beta(t)}^2\cdot\prod_{\Tilde{t}_\beta<\varepsilon\in\suc(\delta)}X_{\varepsilon,\tau_\beta(t)}^2.\] Note that $t_\delta=0$ by definition of the set $U_\beta$ and hence \[\tau_\delta(\tau_\beta(t)+\Tilde{t}_\beta e_\delta)=\tau_\delta(\tau_\beta(t))=\tau_\beta(t).\] This implies that \[X_{\delta,\tau_\beta(t)+\Tilde{t}_\beta e_\delta}=\prod_{\Tilde{t}_\beta<\varepsilon\in\suc(\delta)}X_{\varepsilon,\tau_\delta(\tau_\beta(t)+\Tilde{t}_\beta e_\delta)}^2=\prod_{\Tilde{t}_\beta<\varepsilon\in\suc(\delta)}X_{\varepsilon,\tau_\beta(t)}^2\] is contained in $F[\mathbf{X}]$ and that $X_{\beta,\Tilde{t}}/X_{\beta,t}$ can be written as a product of indeterminates of height less than $\beta$, which proves our claim.\\

Suppose now that $X_{\beta,t}$ divides $X_{\beta,\Tilde{t}}$. Take $r\in U_0$ such that $X_{0,r}$ divides $X_{\beta,t}$. By Lemma 0.3., there is a path from $X_{\beta,t}$ to $X_{0,r}$ and a path from $X_{\beta,\Tilde{t}}$ to $X_{0,r}$. But then \[\tau_\beta(t)=\tau_\beta(r)=\tau_\beta(\Tilde{t})\] and by the first part of the proof it becomes clear that $t_\beta\ge\Tilde{t}_\beta$. \\

\ref{lemma:divide larger:gamma} Suppose first that $\gamma$ is a successor ordinal, say $\gamma=\delta+1$. Let $k$ be a positive integer such that $t_\gamma\neq s_\gamma+k$. Then there is $r\in U_0$ such that there is a path from $X_{\delta,s+ke_\gamma}$ to $X_{0,r}$. Since $(X_{\gamma,s},X_{\delta, s+ke_\gamma})$ is an edge, we obtain by Lemma \ref{lemma:X0s-divides} that $X_{0,r}$ divides $X_{\gamma,s}$. Moreover, by Lemma \ref{lemma:path-X0s}\ref{lemma:path-X0s:taubeta}, we have $\tau_\delta(s+ke_\gamma)=\tau_\delta(r)$ and thus $r_\gamma=s_\gamma+k$. Hence $X_{0,r}$ cannot divide $X_{\beta,t}$, since $\tau_\beta(t)=\tau_\beta(r)$ implies $t_\gamma=r_\gamma$, which contradicts the choice of $k$.

If $\gamma$ is a limit ordinal, then there is $\delta\in\suc(\gamma)$ such that $\beta<\delta$ and $t_\delta=0$ and we find an $r\in U_0$ such that there is a path from $X_{\delta,\tau_\gamma(s)}$ to $X_{0,r}$. This implies that $X_{0,r}$ divides $X_{\gamma,s}$ and by Lemma \ref{lemma:path-X0s}\ref{lemma:path-X0s:sbeta}, we obtain that $r_\delta>0$. Then $X_{0,r}$ cannot divide $X_{\beta,t}$, since $\tau_\beta(t)=\tau_\beta(r)$ implies $r_\delta=t_\delta=0$, which yields a contradiction.
\end{proof}

Before showing the existence of the greatest common divisor of two elements of $\mathbf{X}$, we need another lemma.

\begin{lemma}\label{lemma:limit div}
    Let $\beta,\gamma\in\li(\alpha)$ with $\beta<\gamma$ and let $t\in U_\beta,s\in U_\gamma$ with $\tau_\beta(t)=\tau_\gamma(s)$ and $s_\gamma\le t_\beta$. Then $X_{\beta,t}$ divides $X_{\gamma,s}$ in $F[\mathbf{X}]$.
\end{lemma}
\begin{proof}
We first note that $\tau_\beta(t)+s_\gamma e_\beta\in U_\beta$ since $s_\gamma\leq t_\beta<\beta$. By Lemma \ref{lemma:divide larger}\ref{lemma:divide larger:beta}, $X_{\beta,t}$ divides $X_{\beta,\tau_\beta(t)+s_\gamma e_\beta}$; therefore, it suffices to show that $X_{\beta,\tau_\beta(t)+s_\gamma e_\beta}$ divides $X_{\gamma,s}$. Moreover, since $\tau_\beta(t)=\tau_\beta(\tau_\beta(t)+s_\gamma e_\beta)$, we can assume that $t_\beta=s_\gamma$.

By definition, we have
\[X_{\beta,t}=\prod_{t_\beta<\delta\in\suc(\beta)}X_{\delta,\tau_\beta(t)}^2\]
and 
\[X_{\gamma,s}=\prod_{s_\gamma<\delta\in\suc(\gamma)}X_{\delta,\tau_\gamma(s)}^2.\]
Since $\tau_\beta(t)=\tau_\gamma(s)$ and $s_\gamma= t_\beta$, we can write \[X_{\gamma,s}=X_{\beta,t}\cdot\prod_{\beta<\delta\in\suc(\gamma)}X_{\delta,\tau_\gamma(s)}^2=X_{\beta,t}\cdot X_{\gamma,\tau_\gamma(s)+\beta e_\gamma},\] which proves the claim. 
\end{proof}

\begin{proposition}\label{prop:gcd}
Let $\beta< \gamma<\alpha$ be ordinals and let $t,\Bar{t}\in U_\beta$ and $s\in U_\gamma$.
\begin{enumerate}[(a)]
    \item\label{prop:gcd:beta=gamma} If $X_{\beta,t}$ and $X_{\beta,\Bar{t}}$ have a common divisor in $\mathbf{X}_0$, then \[\gcd(X_{\beta,t},X_{\beta,\Bar{t}})\in\{X_{\beta,t},X_{\beta,\Bar{t}}\}.\] 
    \item\label{prop:gcd:betaSc} If $\beta\in\suc(\alpha)$ and $X_{\beta,t}$ and $X_{\gamma,s}$ have a common divisor in $\mathbf{X}_0$, then \[\gcd(X_{\beta,t},X_{\gamma,s})=X_{\beta,t}.\] 
    \item\label{prop:gcd:betaLim} If $\beta\in \li(\alpha)$ and $X_{\beta,t}$ and $X_{\gamma,s}$ have a common divisor in  $\mathbf{X}_0$, then there is an ordinal $\delta$ with $t_\beta\le\delta<\beta$ such that \[\gcd(X_{\beta,t},X_{\gamma,s})=X_{\beta,\tau_\beta(t)+\delta e_\beta}.\]
\end{enumerate}
\end{proposition}

\begin{proof} 
\ref{prop:gcd:beta=gamma} Let $r\in U_0$ be such that $X_{0,r}$ divides both $X_{\beta,t}$ and $X_{\beta,\Bar{t}}$. By Lemma \ref{lemma:X0s-divides}, there exists a path from $X_{\beta,t}$ to $X_{0,r}$ and a path from $X_{\beta,\Bar{t}}$ to $X_{0,r}$. By Lemma \ref{lemma:path-X0s}\ref{lemma:path-X0s:taubeta}, we obtain \[\tau_\beta(t)=\tau_\beta(r)=\tau_\beta(\Bar{t})\] and the desired result then follows from Lemma \ref{lemma:divide larger}\ref{lemma:divide larger:beta}.

\ref{prop:gcd:betaSc} Let $r\in U_0$ be such that $X_{0,r}$ divides both $X_{\beta,t}$ and $X_{\gamma,s}$. By Lemma \ref{lemma:tree}, there is a unique path from $X_{\gamma,s}$ to $X_{0,r}$: we claim that $X_{\beta,\tau_\beta(t)}$ lies on this path.

Suppose first that there is no $\Tilde{t}\in U_\beta$ such that $X_{\beta,\Tilde{t}}$ lies on the path from $X_{\gamma,s}$ to $X_{0,r}$. Then there are ordinals $\beta_1$ and $\beta_2$ with $\beta_1<\beta<\beta_2$ and elements $t_1\in U_{\beta_1}$ and $t_2\in U_{\beta_2}$ such that the edge $(X_{\beta_2,t_2},X_{\beta_1,t_1})$ lies on the path. This implies that there is a path from $X_{\beta_1,t_1}$ to $X_{0,r}$ and we have $\tau_{\beta_1}(t_1)=\tau_{\beta_1}(r)$ and thus $r_\beta=t_{1,\beta}$. Moreover, by definition of the edges in $\mathcal{T}$, $t_1$ and $t_2$ may differ only in the $\beta_2$-nd component and thus $r_\beta=t_{1,\beta}=t_{2,\beta}=0$. However, since there is a path from $X_{\beta,t}$ to $X_{0,r}$, we also have $r_\beta>0$ by Lemma \ref{lemma:path-X0s}\ref{lemma:path-X0s:sbeta}, which yields a contradiction.

Hence there is a $\Tilde{t}\in U_\beta$ such that $X_{\beta,\Tilde{t}}$ lies on the path from $X_{\gamma,s}$ to $X_{0,r}$ and since $X_{\beta,\Tilde{t}}$ is not the start of the path, we need to have $\Tilde{t}_\beta=0$ by definition. Moreover, we also have $\tau_\beta(\Tilde{t})=\tau_\beta(r)$, which yields that $\Tilde{t}=\tau_\beta(t)$.

By Lemma \ref{lemma:path->divides}, it follows that $X_{\beta,\tau_\beta(t)}$ divides $X_{\gamma,s}$; furthermore, by Lemma \ref{lemma:divide larger}\ref{lemma:divide larger:beta} $X_{\beta,t}$ divides $X_{\beta,\tau_\beta(t)}$. Hence $X_{\beta,t}$ divides $X_{\gamma,s}$ and $\gcd(X_{\beta,t},X_{\gamma,s})=X_{\beta,t}$.

\ref{prop:gcd:betaLim} Let $r\in U_0$ be such that $X_{0,r}$ divides both $X_{\beta,t}$ and $X_{\gamma,s}$. By Lemma \ref{lemma:tree} there is a unique path from $X_{\gamma,s}$ to $X_{0,r}$. Let $\beta_2$ be the smallest limit ordinal with $\beta\le\beta_2$ and with the property that there is $\Tilde{t}\in U_{\beta_2}$ such that $X_{\beta_2,\Tilde{t}}$ lies on this path. Moreover, let $\beta_1$ be the unique successor ordinal such that $(X_{\beta_2,\Tilde{t}},X_{\beta_1,\tau_{\beta_2}(\Tilde{t})})$ is an edge on the path. Note that $\beta_1<\beta$, since otherwise the largest limit ordinal smaller than $\beta_1$ would be equal or bigger than $\beta$.

We distinguish three cases.

\textbf{Case 1:} $\beta_2=\beta$. In this case, $X_{\beta,{\Tilde{t}}}$ is not the start of the path and we have $\Tilde{t}_{\beta}=0$, or equivalently $\tau_{\beta}(\Tilde{t})=\Tilde{t}$, by definition of the edges in $\mathcal{T}.$ Since there is a path from $X_{\beta,t}$ to $X_{0,r}$ and a path from $X_{\beta,\Tilde{t}}$ to $X_{0,r}$, we obtain that \[\Tilde{t}=\tau_\beta(\Tilde{t})=\tau_\beta(r)=\tau_\beta(t).\] Then by Lemma \ref{lemma:divide larger}\ref{lemma:divide larger:beta}, $X_{\beta,t}$ divides $X_{\beta,\Tilde{t}}$ and consequently $X_{\gamma,s}$ as well.

\textbf{Case 2:} $\beta<\beta_2<\gamma$. Again, since $X_{\beta_2,\Tilde{t}}$ is not the start of the path, we have $\Tilde{t}_{\beta_2}=0$, or equivalently $\tau_{\beta_2}(\Tilde{t})=\Tilde{t}$. Since there is a path from $X_{\beta_1,\Tilde{t}}$ to $X_{0,r}$, we obtain $\tau_{\beta_1}(\Tilde{t})=\tau_{\beta_1}(r)$ and consequently  \[\Tilde{t}=\tau_\beta(\Tilde{t})=\tau_\beta(r)=\tau_\beta(t),\] since there is a path from $X_{\beta,t}$ to $X_{0,r}$ and we have $\beta_1<\beta<\beta_2$. Hence $\beta_2$ is a limit ordinal such that $X_{\beta_2,\tau_\beta(t)}$ lies on the path and $t_{\beta_2}=0$. By Lemma \ref{lemma:limit div}, $X_{\beta,t}$ then divides $X_{\beta_2,\tau_\beta(t)}$ and by Lemma \ref{lemma:path->divides}, it divides $X_{\gamma,s}$ as well.

\textbf{Case 3:} $\beta_2=\gamma.$ In this case, $\gamma$ is a limit ordinal and $(X_{\gamma,s},X_{\beta_1,\tau_\gamma(s)})$ is an edge on the path. By the definition of the edges, we have $s_\gamma<\beta_1<\beta$. Again, since there is a path from $X_{\beta_1,\tau_\gamma(s)}$ to $X_{0,r}$ and a path from $X_{\beta,t}$ to $X_{0,r}$, we obtain that $\tau_{\beta_1}(\tau_\gamma(s))=\tau_{\beta_1}(r)$, which implies that 
\begin{equation*}
\tau_\beta(\tau_\gamma(s))=\tau_\gamma(s)=\tau_\beta(r)=\tau_\beta(t).
\end{equation*}
Hence, $t_\delta=0$ if $\beta<\delta\leq\gamma$, and so
\begin{equation*}
t=\tau_\beta(t)+t_\beta e_\beta=\tau_\gamma(s)+t_\beta e_\beta.
\end{equation*}
If $s_\gamma\le t_\beta$, then, since $t_\beta<\beta$ and using Lemma \ref{lemma:limit div}, we have that $X_{\beta,t}$ divides $X_{\gamma,s}$. Suppose that $s_\gamma>t_\beta$. Since $s_\gamma<\beta$,  we have $\tau_\beta(t)+s_\gamma e_\beta\in U_\beta$ and similarly, $X_{\beta,\tau_\beta(t)+s_\gamma e_\beta}$ divides $X_{\gamma,s}$. By Lemma \ref{lemma:divide larger}\ref{lemma:divide larger:beta}, $X_{\beta,\tau_\beta(t)+s_\gamma e_\beta}$ also divides $X_{\beta,t}$, and thus the quotient
\begin{equation*}
Y:=X_{\beta,t}/X_{\beta,\tau_\beta(t)+s_\gamma e_\beta}=\prod_{t_\beta<\varepsilon\in\suc(s_\gamma+1)} X_{\varepsilon,\tau_\beta(t)}^2.
\end{equation*}
belongs to $F[\XX]$. To show that $X_{\beta,\tau_\beta(t)+s_\gamma e_\beta}$ is the greatest common divisor we are looking for, we need to show that $Y$ and $X_{\gamma,s}$ are coprime, and it is enough to show that they do not have a common divisor in  $\mathbf{X}_0$.

Suppose that $r'\in U_0$ is such that $X_{0,r'}$ divides $X_{\gamma,s}$. Then, there is a successor ordinal $\beta'>s_\gamma$ such that $(X_{\gamma,s},X_{\beta',\tau_\gamma(s)})$ lies on the unique path from $X_{\gamma,s}$ to $X_{0,r'}$. By Lemma \ref{lemma:path-X0s}\ref{lemma:path-X0s:sbeta}, it follows that $r'_{\beta'}>0$. If $X_{0,r'}$ also divides $Y$, then there is $\delta\in\suc(s_\gamma+1)$, such that $X_{0,r'}$ divides $X_{\delta,\tau_\beta(t)}$ and $t_\beta<\delta$. Then $\tau_\delta(\tau_\beta(t))=\tau_\delta(r')$. Since $\delta< \beta'$, we have $\tau_\beta(t)_{\beta'}=r'_{\beta'}>0$. This implies that $\beta'>\beta$, since $t\in U_\beta$. However, since we have $\tau_\beta(\tau_\gamma(s))=\tau_\beta(t)$ and $\beta<\beta'<\gamma$, we obtain that $t_{\beta'}=\tau_\beta(t)_{\beta'}=0$, which yields a contradiction. 
\end{proof}

We will now generalize this result for monomials. 

\begin{lemma}\label{lemma:gcd mult}
Let $D$ be an integral domain and $a,b,c\in D$ be nonzero. If $\gcd(a,c)$ exists and $\gcd(b,c)=1$ then $\gcd(ab,c)=\gcd(a,c)$.
\end{lemma}
\begin{proof}
Clearly $\gcd(a,c)$ divides $\gcd(ab,c)$. Conversely, $\gcd(ab,c)$, divides $c$ and thus it divides $ac$; hence $\gcd(ab,c)|\gcd(ab,ac)=a\gcd(b,c)=a$. Therefore $\gcd(ab,c)|\gcd(a,c)$. The claim is proved.
\end{proof}

\begin{proposition}\label{prop:gcd monomial}
    Let $X=X_1\ldots X_n$ and $Y=Y_1\ldots Y_m$ be monomials with $X_1\ldots X_n,Y_1\ldots Y_m\in\mathbf{X}.$ Then $X$ and $Y$ have a greatest common divisor.
\end{proposition} 

\begin{proof} By Lemma \ref{lemma:product}, we can write  \[X=\prod_{r\in U_0}X_{0,r}^{n_r},\] and \[Y=\prod_{r\in U_0}X_{0,r}^{m_r},\] where $n_r,m_r\in \mathbb{N}_0$. In particular, $X$ (resp. $Y$) is divisible by $X_{0,r}$ with multiplicity $n_r$ (resp. $m_r$) for every $r\in U_0$. Suppose that the $\gcd$ of $X$ and $Y$ exists. It is then easy to see that \[\gcd(X,Y)=\prod_{r\in U_0}X_{0,r}^{\min(n_r,m_r)}.\] Hence it is enough to show that there is a monomial $Z$ such that $X/Z$ and $Y/Z$ do not share any common divisors in $\mathbf{X}_0$. \\

Let now $X_i=X_{\beta_i,t_i}$ for every $i$. We proceed by induction on the maximal height $\beta$ of the $X_i$'s.

If $\beta=0$, then $n_r\neq 0$ for only finitely many $r$ and we obtain \[\gcd(X,Y)=\prod_{r\in U_0}X_{0,r}^{\min(n_r,m_r)}.\]  

Suppose that the claim is true if the maximum is $\beta_i<\beta$ for every $i$.

Suppose first that $X=X_{\beta,t}\in\mathbf{X}_\beta$ is an indeterminate. Write $Y_j=X_{\delta_j,s_j}$ for some ordinal numbers $\delta_j$ and some $s_j\in U_{\delta_j}$. If there is a $j$ such that $\gcd(X,Y_j)=1$, then \[\gcd(X,Y)=\gcd(X,Y/Y_j)\] by Lemma \ref{lemma:gcd mult} and thus we can assume that $\gcd(X,Y_i)\neq 1$ for every $i$. If $\delta_j\ge\beta$ for some $j$, then by Proposition \ref{prop:gcd} we have $\gcd(X,Y_j)=X_{\beta,\tau_\beta(t)+\delta e_\beta}$, for some $\delta\ge t_\beta$. Then, by Lemma \ref{lemma:divide larger}\ref{lemma:divide larger:beta}, $X/X_{\beta,\tau_\beta(t)+\delta e_\beta}$ can be written as a (finite) product of indeterminates of height less than $\beta$. Hence the claim follows by induction. 

Suppose thus that $\delta_j<\beta$ for all $j$. Write
\begin{equation*}
Y=\gcd(X,Y_1)\cdots \gcd(X,Y_m) Z_1\ldots Z_m,
\end{equation*}
where $Z_i=Y_i/\gcd(X,Y_i)$. By Proposition \ref{prop:gcd} we have either $\gcd(X,Y_i)=Y_i$, in which case $Z_i=1$, or $\delta_i$ is a limit ordinal and $\gcd(X,Y_i)=X_{\delta_i,\tau_{\delta_i}(s_i)+\varepsilon e_{\delta_i}}$ for some ordinal $\varepsilon>(s_i)_{\delta_i}$; in the latter case, by  case 3 of the proof of Proposition \ref{prop:gcd}\ref{prop:gcd:betaLim} we have $\gcd(X,Z_i)=1$. Hence $\gcd(X,Z_i)=1$ in each case and $\gcd(X,Z_1\cdots Z_m)=1$, which means that
\begin{equation*}
\gcd(X,Y)=\gcd(X,\gcd(X,Y_1)\cdots \gcd(X,Y_m))
\end{equation*}
or, equivalently, we can suppose that $Y_i$ divides $X$ for each $i$.

Suppose now that $\beta=\gamma+1$ is a successor ordinal. Set $k_j:=(s_j)_\beta$, $k:=\underset{1 \leq j \leq m}{\max}k_j$  We claim that $\gcd(X_{\beta,t+ke_\beta},Y)=1$. Indeed, suppose that there is an $r\in U_0$ such that $X_{0,r}$ divides both $X_{\beta,t+ke_\beta}$ and some $Y_i=X_{\delta_i,s_i}$. 
By Lemma \ref{lemma:path-X0s}, we obtain $r_\beta>(t+ke_\beta)_\beta\ge k$ and $\tau_{\delta_i}(s_i)=\tau_{\delta_i}(r)$, which implies that \[r_\beta=(s_i)_\beta>k.\] However, this contradicts the choice of the integer $k$. Thus we have  $\gcd(X_{\beta,t+ke_\beta},Y)=1$. Therefore, using the definition of $X_{\beta,t}$ we have
\begin{equation*}
\begin{aligned}
\gcd(X,Y)=\gcd(X_{\beta,t},Y)&=\gcd(X_{\beta,t+ke_\beta}X_{\gamma,t+e_\beta}^2\cdots X_{\gamma,t+ke_\beta}^2,Y)=\\
&=\gcd(X_{\gamma,t+i_\beta}^2\cdots X_{\gamma,t+ke_\beta}^2,Y)
\end{aligned}
\end{equation*}
and the claim follows by induction.

Suppose now that $\beta$ is a limit ordinal. Set \[\gamma_j:=\max\{\delta\in\suc(\beta):(s_j)_\delta\neq 0\}\] and \[\gamma:=\max\{\delta_1,\ldots,\delta_m,\gamma_1,\ldots,\gamma_m,t_\beta\}.\] Note that these maxima exist, since every $s_j$ has only finitely many nonzero components. Moreover, we have $\gamma<\beta$ and Lemma \ref{lemma:divide larger}\ref{lemma:divide larger:beta} implies that we can write \[X_{\beta,t}=X_{\beta,\tau_\beta(t)+\gamma e_\beta}\cdot \Tilde{X},\] where $\Tilde{X}$ can be written as a product of indeterminates of height less than $\beta$. By induction hypothesis, it is thus enough to show that $\gcd(X_{\beta,\tau_\beta(t)+\gamma e_\beta},Y)=1$. Suppose that there is $r\in U_0$ such that $X_{0,t}$ divides both $X_{\beta,\tau_\beta(t)+\gamma e_\beta}$ and $Y_i$ for some $i$. By Lemma 0.7, there is $\varepsilon\in\suc(\beta)$ with $\gamma<\varepsilon$ such that $X_{0,r}$ divides $X_{\varepsilon,\tau_\beta(t)}$. Then Lemma \ref{lemma:path-X0s} yields $r_\varepsilon>0$ and $\tau_{\delta_i}(s_i)=\tau_{\delta_i}(r)$ and consequently $r_\varepsilon=(s_i)_\varepsilon>0$, since $\delta_i<\varepsilon$. However, by definition of $\gamma_i$, this implies that $\varepsilon\le\gamma_i\le\gamma$, which provides the final contradiction.

Hence $\gcd(X,Y)$ exists if $X=X_{\beta,t}\in\mathbf{X}_\beta$. Suppose now that $X=X_1\cdots X_n$ with $X_i=X_{\beta_i,t_i}$ and $\beta_i\leq\beta$ for each $i$; we proceed by induction on $n$. The case $n=1$ is the previous one. Suppose $n>1$. Then, $Z:=\gcd(X_n,Y)$ exists, and $\gcd(X_n/Z,Y/Z)=1$; hence,
\begin{equation*}
\begin{aligned}
\gcd(X_1\cdots X_n,Y) &=\gcd\left(X_1\cdots X_{n-1}\frac{X_n}{Z}Z,\frac{Y}{Z}Z\right)=\\
&=Z\cdot\gcd\left(X_1\cdots X_{n-1}\frac{X_n}{Z},\frac{Y}{Z}\right)=\\
&=Z\cdot\gcd\left(X_1\cdots X_{n-1},\frac{Y}{Z}\right).
\end{aligned}
\end{equation*}
Since $Y/Z$ is a monomial, $\gcd\left(X_1\cdots X_{n-1},\frac{Y}{Z}\right)$ exists by induction and so does $\gcd(X_1\cdots X_n,Y)$. This concludes the proof that $\gcd(X,Y)$ exists when the maximum of the $\beta_i$ is $\beta$. By induction on $\beta$, $\gcd(X,Y)$ exists for every monomial $X$.
\end{proof} 

\section{The example}\label{sect:example}
We are now ready to define the domain that will be our example.

Let $S$ be the set of all elements of $F[\mathbf{X}]$ which are not divisible by a proper monomial. By Lemma \ref{lemma:X0s-prime}\ref{lemma:X0s-prime:prime}, every element of $\mathbf{X}_0$ is a prime element and thus $S$ is a multiplicative subset of $F[\mathbf{X}]$. We set \[D:=S^{-1}F[\mathbf{X}].\]

\begin{lemma}
Let $D$ be as above. Then, $D$ is a GCD-domain and every element of $D$ is associated to a monomial.
\end{lemma}
\begin{proof}
Let $x\in D$: then, there is an $s\in S$ such that $sx\in F[\XX]$. Thus, we can write $sx=a_1Y_1+\cdots+a_nY_n$ where $a_i\in F$ and each $Y_i$ is a monomial. By Proposition \ref{prop:gcd monomial}, $Z:=\gcd(Y_1,\ldots,Y_n)$ exists and is a monomial. Then,
\begin{equation*}
sx=\sum_{i=1}^{n}a_iY_i=Z\sum_{i=1}^{n}a_i(Y_i/Z);
\end{equation*}
by construction, $\gcd(Y_1/Z,\ldots,Y_n/Z)=1$, and thus the sum of the rightmost hand side belongs to $S$, i.e., it is a unit of $D$. Therefore, $x$ is associated in $D$ to the monomial $Z$.

If $x,y\in D$ are monomials, then $\gcd(x,y)$ exists in $D$ since it exists in $F[\XX]$; by the previous part of the proof, it follows that $\gcd(x,y)$ exists for every pair $x,y$ of elements of $D$. Hence $D$ is a GCD-domain.
\end{proof}

We now want to find the prime ideals of $D$, and to do so, we explicitly define the following ideals indexed by ordinal numbers $\beta<\alpha$ and $t\in U_\beta$.
\begin{itemize}
\item For each $t\in U_0$, we define \[M_{0,t}:=X_{0,t}D.\]
\item If $\beta\in\suc(\alpha)$ and $t\in U_\beta$, we set \[M_{\beta,t}:=\sum_{i=0}^{\infty}X_{\beta,\tau_\beta(t)+i e_\beta}D.\]
\item If $\beta\in\li(\alpha)$ and $t\in U_\beta$, we set \[M_{\beta,t}:=\sum_{\gamma<\beta}X_{\beta,\tau_\beta(t)+\gamma e_\beta}D.\]
\end{itemize}

For a fixed ordinal $\beta$, we denote the collection of all ideals $M_{\beta,t}$, where $t\in U_\beta$, by $\mathcal{M}_\beta$ and we write \[\mathcal{M}:=\bigcup_{\beta<\alpha}\mathcal{M}_\beta.\] 

\begin{remark}
~\begin{enumerate}
\item By Lemma \ref{lemma:divide larger}\ref{lemma:divide larger:beta}, every element of $M_{\beta,t}$ is divisible by some $X_{\beta,s}$, where $\tau_\beta(s)=\tau_\beta(t)$. In other words, every finitely generated ideal $I\subseteq M_{\beta,t}$ is contained in $X_{\beta,s}D$ for some $s\in U_\beta$ with $\tau_\beta(s)=\tau_\beta(t)$.
\item It is possible for $M_{\beta,t}$ to be equal to $M_{\beta,t'}$ even if $t\neq t'$. Indeed, $M_{\beta,t}=M_{\beta,t+e_\beta}$ for every $\beta$. More precisely, by Lemma \ref{lemma:divide larger}, we have $M_{\beta,t}=M_{\gamma,s}$ if and only if $\beta=\gamma$ and $\tau_\beta(t)=\tau_\beta(s)$. 
\end{enumerate}
\end{remark}

\begin{proposition}\label{prop:max ideals}
    We have $\Spec(D)=\mathcal{M}\cup\{0\}$.
\end{proposition}

\begin{proof}
We first show that every element of $\mathcal{M}$ is a maximal ideal. Take $\beta<\alpha$ and $t\in U_\beta$, and suppose that there is a maximal ideal $M$ such that $M_{\beta,t}\subsetneq M$. Let $X\in M\setminus M_{\beta,t}$. Since every element of $D$ is associated to a monomial and $M$ is prime, we can assume without loss of generality that $X\in\mathbf{X}$, say $X=X_{\gamma,s}$, where $s\in U_\gamma$. Then $X_{\beta,t}+X_{\gamma,s}\in M$; in particular, $X_{\beta,t}+X_{\gamma,s}\notin S$, and thus there must be a monomial dividing $X_{\beta,t}+X_{\gamma,s}$. Therefore, $X_{\beta,t}$ and $X_{\gamma,s}$ share a common divisor in $\mathbf{X}_0$.

We distinguish several cases.

If $\beta=0$, then it follows that $X_{\beta,t}$ divides $X_{\gamma,s}$ and so $X_{\gamma,s}\in M_{\beta,t}$, a contradiction.

If $0<\beta<\gamma$, then by Proposition \ref{prop:gcd}, there is a $\delta$ such that $\gcd(X_{\beta,t},X_{\gamma,s})=X_{\beta,\tau_\beta(t)+\delta e_\beta}$; since the latter is one of the generators of $M_{\beta,t}$ it follows that $X_{\gamma,s}\in M_{\beta,t}$, again a contradiction.

If $\beta=\gamma$, then by Proposition \ref{prop:gcd}\ref{prop:gcd:beta=gamma}, $X_{\beta,t}$ divides $X_{\gamma,s}$ or $X_{\gamma,s}$ divides $X_{\beta,t}$. By Lemma \ref{lemma:divide larger}\ref{lemma:divide larger:beta}, we obtain that $\tau_\beta(t)=\tau_\beta(s)$ and hence $X_{\gamma,s}\in M_{\beta,t}$, another contradiction. 

Suppose now that $\gamma<\beta$. We distinguish two subcases.

We first consider the case in which $\beta$ is a successor ordinal. If $X_{0,r}$ divides $X_{\gamma,s}$, then by Lemma \ref{lemma:path-X0s}\ref{lemma:path-X0s:taubeta}, we have $\tau_\gamma(s)=\tau_\gamma(r)$, and in particular $s_\beta=r_\beta$. By Lemma \ref{lemma:path-X0s}\ref{lemma:path-X0s:sbeta}, $X_{0,r}$ cannot divide $X_{\beta,\tau_\beta(t)+s_\beta e_\beta}$, and thus $\gcd(X_{\gamma,s},X_{\beta,\tau_\beta(t)+s_\beta e_\beta})=1$. However, since $X_{\beta,\tau_\beta(t)+s_\beta e_\beta}$ is one of the generators of $M_{\beta,t}$, by the same reasoning as above, the greatest common divisor shouldn't be $1$; hence we get a contradiction.

Suppose now that $\beta$ is a limit ordinal. Since $s$ has finitely many nonzero components, there is a $\delta$ such that $\gamma<\delta<\beta$ and $s_\varepsilon=0$ for every $\delta<\varepsilon<\beta$. The element $X_{\beta,\tau_\beta(t)+\delta e_\beta}$ is one of the generators of $M_{\beta,t}$; by the above reasoning, there is an $r\in U_0$ such that $X_{0,r}$ divides both $X_{\gamma,s}$ and $X_{\beta,\tau_\beta(t)+\delta e_\beta}$. Then $\tau_\gamma(s)=\tau_\gamma(r)$ by Lemma \ref{lemma:path-X0s}\ref{lemma:path-X0s:taubeta} and thus $r_\varepsilon=0$ for every $\delta<\varepsilon<\beta$.

By Lemma \ref{lemma:X0s-divides}, there is a path from $X_{\beta,\tau_\beta(t)+\delta e_\beta}$ to $X_{0,r}$. Since
\[\tau_\beta(\tau_\beta(t)+\delta e_\beta)=\tau_\beta(t),\]
there is a successor ordinal $\varepsilon>\delta$, such that the path passes through $X_{\varepsilon,\tau_\beta(t)}$. However, by Lemma \ref{lemma:path-X0s}\ref{lemma:path-X0s:taubeta} this implies that $r_\varepsilon>0$, which contradicts the choice of $\delta$. Hence \[\gcd(X_{\gamma,s},X_{\beta,\tau_\beta(t)+\delta e_\beta})=1\] and as before, this contradicts our initial assumption.
Therefore, each $M_{\beta,t}\in\mathcal{M}$ is a maximal ideal.

\bigskip

Let now $P$ be a nonzero prime ideal of $D$. Then, $P$ must contain a monomial, and since $P$ is prime we obtain $P\cap \mathbf{X}\neq\emptyset$. Let $\beta$ be the least ordinal such that $P\cap\mathbf{X}_\beta\neq\emptyset$ and let $t\in U_\beta$ such that $X_{\beta,t}\in P$. 

If $\beta=0$, then $M_{0,t}\subseteq P$ and so $M_{0,t}=P$. If $\beta>0$, we claim that $P$ must contain $X_{\beta,t+\delta e_\beta}$ for all $\delta$. Indeed, $X_{\beta,t}=Y\cdot X_{\beta,t+\delta e_\beta}$ for some monomial $Y$ that is a product of elements in $\bigcup_{\gamma<\beta}\mathbf{X}_\gamma$ (see the proof of Lemma \ref{lemma:divide larger}\ref{lemma:divide larger:beta}); by the definition of $\beta$, we have $Y\notin P$ and thus $X_{\beta,t+\delta e_\beta}\in P$. Therefore, $P$ contains the union $\bigcup_\delta X_{\beta,t+\delta e_\beta}D$, which is simply $M_{\beta,t}$. By the previous part of the proof, $M_{\beta,t}$ is a maximal ideal, and thus $M_{\beta,t}=P$. Hence all nonzero prime ideals of $D$ are in $\mathcal{M}$, and $\Max(D)=\mathcal{M}$.
\end{proof}

\begin{theorem}\label{theorem:sp-rank alpha}
$D$ is an almost Dedekind domain of SP-rank $\alpha$.
\end{theorem}
\begin{proof}
Let $\beta<\alpha$ and let $t\in U_\beta$: by Proposition \ref{prop:max ideals}, we need  to show that each $D_{M_{\beta,t}}$ is a DVR. Note that every indeterminate of height less than $\beta$ becomes a unit in $D_{M_{\beta,t}}$ and thus all the generators of $M_{\beta,t}$ are associated to each other; therefore, $M_{\beta,t}D_{M_{\beta,t}}$ is a principal ideal. Moreover, $D_{M_{\beta,t}}$ is one-dimensional since the only prime ideal contained in $M_{\beta,t}$ is $(0)$ (Proposition \ref{prop:max ideals}); hence $D_{M_{\beta,t}}$ is a DVR and $D$ is an almost Dedekind domain.

\bigskip

We now want to show that $\Crit^\beta(D)=\{M_{\gamma,t}\mid \gamma\geq\beta\}$ for every $\beta\geq 1$. We claim both that this is true and that $X_{\gamma,s}$ becomes a unit in $T_\beta$ if $\gamma<\beta$. We proceed by induction on $\beta$.

If $\beta=1$, then we need to show that $M_{\gamma,t}$ is critical if and only if $\gamma\geq 1$. If $\gamma=0$ then $M_{0,t}$ is principal and thus it is not critical. Otherwise, let $I\subseteq M_{\gamma,t}$ be a finitely generated ideal. Then, $M_{\gamma,t}$ is the union of the increasing family of principal ideals $\{X_{\gamma,\tau_\gamma(t)+ie_\gamma}D\}_{i=0}^\infty$, and thus $I\subseteq X_{\gamma,s}D$ for some $s\in U_\gamma$. There is an edge $(X_{\gamma,s},X_{\delta,s'})$: then, by Lemma \ref{lemma:path->divides}, $X_{\delta,s'}^2$ divides $X_{\gamma,s}$ and thus $I\subseteq M_{\delta,s'}^2$. Since $I$ was arbitrary, $M_{\gamma,t}$ is critical. Moreover, each $X_{0,s}$ becomes a unit in $T_1$.

Suppose now that the claim holds for all ordinal numbers up to $\beta$. If $\beta$ is a limit ordinal the claim is trivial.

Suppose that $\beta$ is a successor ordinal, say $\beta=\lambda+1$. By definition, the maximal ideals of $T_\lambda$ are the extensions of the elements of $\Crit^\lambda(D)$, which by the inductive step are the $M_{\gamma,s}$ for $\gamma\geq\lambda$; moreover,  each $X_{\gamma,s}$ of height less than $\lambda$ is a unit in $T_\lambda$. In particular, the generators $X_{\lambda,\tau_\lambda(t)+\delta e_\lambda}$ are all associated in $T_\lambda$. Therefore, $M_{\lambda,t}T_\lambda$ is a principal ideal and hence not critical, which implies that $X_{\lambda,t}$ becomes invertible in $T_{\lambda+1}=T_\beta$. Conversely, if $\gamma>\lambda$, take a finitely generated ideal $I\subseteq M_{\gamma,t}T_\lambda$. As in the case $\beta=0$, we have that $I\subseteq X_{\gamma,s}T_\lambda$ for some $s$, and we can find a $\delta\geq\beta$ such that $(X_{\gamma,s},X_{\delta,s'})$ is an edge, and thus $X_{\delta,s'}^2$ divides $X_{\gamma,s}$. Since $X_{\delta,s'}$ is not a unit in $T_\lambda$, $M_{\gamma,t}T_\lambda$ is therefore critical, and $\Crit(T_\lambda)=\{M_{\gamma,t}\mid \gamma>\lambda\}$. Hence,
\begin{equation*}
\Crit^\beta(D)=\Crit^{\lambda+1}(D)=\{M_{\gamma,t}\mid \gamma>\lambda\}=\{M_{\gamma,t}\mid \gamma\geq\beta\}.
\end{equation*}
By induction, this equality holds for every $\beta$.

In particular, $\Crit^\beta(D)=\emptyset$ if and only if $\beta\geq\alpha$. Hence, $D$ has SP-rank $\alpha$.
\end{proof}

\begin{corollary}
    Let $\alpha$ be an arbitrary ordinal number. Then there is an almost Dedekind domain $D$ of SP-rank $\alpha$.
\end{corollary} \qed

\begin{example}
We work out the case $\alpha=3$ in an explicit way.

In this case, we only have indeterminates of three heights, $0$, $1$ and $2$, and the index sets can be described as
\begin{equation*}
\begin{aligned}
U_0=U_1& = \{(t_1,t_2)\mid t_1,t_2\in\insN_0\};\\
U_2 &= \{(0,t_2)\mid t_2\in\insN_0\}.
\end{aligned}
\end{equation*}
To simplify the notation, we set
\begin{equation*}
\begin{aligned}
Y_{a,b} & := X_{0,(a,b)},\\
Z_{a,b} & := X_{1,(a,b)},\\
W_n & := X_{2,(0,n)}.
\end{aligned}
\end{equation*}
Then, the infinite products take the form
\begin{equation*}
Z_{a,b}=\prod_{i=1}^\infty Y_{a+i,b}^2\quad\text{~and~}\quad W_n=\prod_{i=1}^\infty Z_{0,n+i}^2.
\end{equation*}
The maximal ideals are those of the form $M_{0,t}$, $M_{1,t}$ and $M_{2,t}$. The first ones are principal:
\begin{equation*}
M_{0,(a,b)}=Y_{a,b}D.
\end{equation*}
If $t=(a,b)\in U_1$, then $\tau_1(t)=(0,b)$. Thus
\begin{equation*}
M_{1,(a,b)}=\sum_{n=0}^\infty Z_{(n,b)}D=\bigcup_{n=0}^\infty Z_{(n,b)}D.
\end{equation*}
In particular, $M_{1,(a,b)}$ does not depend on $a$: these maximal ideals are indexed by $b$. We write $M_{1,b}:=M_{1,(0,b)}$. Lastly, $\tau_2((0,n))=(0,0)$, and so
\begin{equation*}
M_{2,n}=\sum_{i=0}^\infty W_{n+i}D=\bigcup_{i=0}^\infty W_{n+i}D.
\end{equation*}
Therefore, $M_{2,n}=M_{2,0}$, and thus there is a unique maximal ideal in this third class, which we denote by $M_2$.

The maximal ideals $M_{0,(a,b)}$ are principal and thus not critical. On the other hand, both the ideals $M_{1,b}$ and $M_2$ are critical: indeed, the chain $\{Z_{n,b}D\}$ is increasing and $Z_{n,b}$ belongs to $Y_{n+1,b}^2D=M_{0,(n+1,b)}^2$ for each $n$, and likewise $\{W_{n}D\}$ is increasing and $W_{n}\in Z_{0,n+1}^2D\subseteq M_{1,n+1}^2$. 

The ring $T_1$ is obtained from $D$ by inverting all the $Y_{a,b}$; since $Z_{a,b}=Y_{a+1,b}^2\cdot Z_{a+1,b}$, it follows that $Z_{a,b}$ and $Z_{a',b}$ are associated in $T_1$ for all $a,a'$, and so $M_{1,b}T_1$ is generated by any $Z_{a,b}$, and thus is principal. A similar reasoning to above shows that $M_2T_1$ remains critical; repeating the process, we obtain that $T_2$ has a unique maximal ideal, $M_2T_2$, which thus must be principal. 

It follows that the critical sets are:
\begin{equation*}
\begin{aligned}
\Crit(D)=\Crit^1(D) = & \{M_{1,a}\mid a\in\insN_0\}\cup\{M_2\},\\
\Crit^2(D) = & \{M_2\},\\
\Crit^3(D) = & \emptyset.
\end{aligned}
\end{equation*}
Hence $D$ has SP-rank $3$.
\end{example} 

\section{Generalizations}
Let $I$ be an ideal of an almost Dedekind domain $D$. The \emph{ideal function} associated to $I$ is
\begin{equation*}
\begin{aligned}
\nu_I\colon\Max(D) & \longrightarrow\insN\\
M & \longmapsto v_M(I)
\end{aligned}
\end{equation*}
where $v_M$ is the valuation relative to the localization $D_M$. Then, a maximal ideal $M$ is critical if there is no finitely generated ideal $I\subseteq M$ such that $\sup\nu_I=1$.

Extending this definition, following \cite{boundness} we say that:
\begin{itemize}
\item $M$ is \emph{$n$-critical} (for $n\in\insN$) if there is no finitely generated ideal $I\subseteq M$ such that $\sup\nu_I\leq n$;
\item $M$ is \emph{$\omega$-critical} if  there is no finitely generated ideal $I\subseteq M$ such that $\nu_I$ is bounded.
\end{itemize}

If we denote by $\phantom{}^n\Crit(D)$ the set of $n$-critical maximal ideals of $D$ (with $n\in\insN\cup\{\omega\}$, we have $\Crit(D)=\phantom{}^1\Crit(D)$ and there is a chain
\begin{equation*}
\Crit(D)\supseteq\phantom{}^2\Crit(D)\supseteq\cdots\supseteq\phantom{}^n\Crit(D)\supseteq\cdots\supseteq\phantom{}^\omega\Crit(D)
\end{equation*}
Moreover, we can construct a sequence analogous to $\Crit^\beta(D)$ by setting:
\begin{itemize}
\item $\phantom{}^nT_0:=D$, $\phantom{}^n\Crit^0(D):=\Max(D)$;
\item if $\alpha=\beta+1$ is a successor ordinal, then
\begin{equation*}
\phantom{}^n\Crit^\alpha(D):=\{P\in\Max(D)\mid P(\phantom{}^nT_\beta)\in\phantom{}^n\Crit(\phantom{}^nT_\beta)\};
\end{equation*}
\item if $\alpha$ is a limit ordinal, then
\begin{equation*}
\phantom{}^n\Crit^\alpha(D):=\bigcap_{\beta<\alpha}\phantom{}^n\Crit^\beta(D);
\end{equation*}
\item $\displaystyle{\phantom{}^nT_\alpha:=\bigcap\{D_M\mid M\in\phantom{}^n\Crit^\alpha(D)\}}$.
\end{itemize}

Let now $D$ be the domain we constructed in Section \ref{sect:example}. An inspection of the proof of Theorem \ref{theorem:sp-rank alpha} shows that no maximal ideal of $D$ is $2$-critical, since for every maximal ideal $M_{\beta,t}$ there is an $X=X_{\gamma,t'}$ with $\gamma<\beta$ such that $\nu_{XD}(M_{\gamma,t'})=2$. Hence, $\phantom{}^n\Crit(D)=\emptyset$ for every $n\geq 2$.

On the other hand, suppose that we had defined
\begin{equation*}
X:=\prod_{(X,Y)\in\mathcal{T}}Y^3,
\end{equation*}
that is, that we used $Y^3$ instead of $Y^2$. Then, essentially nothing in our construction would change: we can still find that any two monomials have a greatest common divisor, we can define $D$, and $D$ will be an almost Dedekind domain with SP-rank $\alpha$. However, this time every critical ideal would also be $2$-critical, while no ideal would be $3$-critical; more precisely, we would have
\begin{equation*}
\phantom{}^2\Crit^\beta(D)=\Crit^\beta(D)
\end{equation*}
for every ordinal number $\beta<\alpha$, while $\phantom{}^k\Crit^\beta(D)=\emptyset$ if $k>2$. Similarly, using $Y^n$ instead of $Y^2$ would yield an almost Dedekind domain such that
\begin{equation*}
\Crit^\beta(D)=\phantom{}^2\Crit^\beta(D)=\cdots=\phantom{}^{n-1}\Crit^\beta(D)
\end{equation*}
while $\phantom{}^n\Crit^\beta(D)=\emptyset$ for all $\beta<\alpha$.

We can say even more. Suppose that $\alpha=\omega$, and define
\begin{equation*}
X:=\prod_{(X,Y)\in\mathcal{T}}Y^{n+1},
\end{equation*}
where $n$ is equal to the height of $X$. Even in this case, the construction of the ring $D$ and the description of its maximal ideals works as in the base case, and $\Crit(D)=\{M_{\beta,t}\mid \beta\geq 1\}$.

Consider now $M=M_{\beta,t}$. If $\beta=1$, then, $M$ is not $2$-critical, since $\nu_{X_{0,t+e_1}D}(M)=2$. On the other hand, if $\beta\geq 2$ and $Y=X_{\beta-1,t'}$ is a divisor of $X_{\beta,t}$, then $\nu_{YD}(M)\ge 3$; since this happens for all divisors, we have that $M_{\beta,t}$ is $2$-critical. Thus $\phantom{}^2\Crit(D)=\{M_{\beta,t}\mid \beta\geq 2\}$.

An analogous reasoning shows that 
\begin{equation*}
\phantom{}^n\Crit(D)=\{M_{\beta,t}\mid \beta\geq n\}
\end{equation*}
and that
\begin{equation*}
\phantom{}^\omega\Crit(D)=\{M_{\beta,t}\mid \beta\geq \omega\}=\emptyset;
\end{equation*}
in particular, the almost Dedekind domain $D$ satisfies
\begin{equation*}
\Crit(D)\supsetneq\phantom{}^2\Crit(D)\supsetneq\cdots\supsetneq\phantom{}^n\Crit(D)\supsetneq\cdots\supsetneq\phantom{}^\omega\Crit(D).
\end{equation*}

\bibliographystyle{plain}
\bibliography{biblio}
\end{document}